\documentclass{amsart}
\usepackage{mathrsfs}

\theoremstyle{plain}
\newtheorem{thm}{Theorem}[section]
\newtheorem{lemma}[thm]{Lemma}
\newtheorem{cor}[thm]{Corollary}
\newtheorem{prop}[thm]{Proposition}

\theoremstyle{definition}

\newtheorem{remark}[thm]{Remark}

\catcode`\@=11
\def\mequal{\mathrel{\mathpalette\@mvereq{\hbox{\sevenrm m}}}}
\def\@mvereq#1#2{\lower.5\p@\vbox{\baselineskip\z@skip\lineskip1.5\p@
    \ialign{$\m@th#1\hfil##\hfil$\crcr#2\crcr=\crcr}}}
\def\partr#1#2{/\kern-.08333em/_{#1,#2}^{\phantom{.}}}
\def\invpartr#1#2{/\kern-.08333em/_{#1,#2}^{-1}}
\def\hpartr#1#2{/\kern-.08333em/_{#1,#2}^{h}}
\def\Epartr#1#2{/\kern-.08333em/_{#1,#2}^{E}}

\def\newdot{{\kern.8pt\cdot\kern.8pt}}
\def\,{\relax\ifmmode\mskip\thinmuskip\else\thinspace\fi}
\def\{{\relax\ifmmode\lbrace\else $\lbrace$\fi}
\def\}{\relax\ifmmode\rbrace\else $\rbrace$\fi}

\font\sevenrm=cmr7

\newcommand\NN{\mathbb{N}}

\newcommand\RR{\mathbb{R}}

\newcommand\PP{\mathbb{P}}

\newcommand{\SC}{{\mathscr C}}

\def\mathpal#1{\mathop{\mathchoice{\text{\rm #1}}%
   {\text{\rm #1}}{\text{\rm #1}}%
   {\text{\rm #1}}}\nolimits}

\def\grad{\mathpal{grad}}
\def\id{\mathpal{id}}

\def\grad{\mathop{\rm grad}\nolimits}
\def\di{\displaystyle}
\def\f{\frac}

\def\b{\beta }
\def\D{\Delta }
\def\d{\delta }
\def\e{\varepsilon }

\def\g{\gamma }

\def\n{\nabla }

\def\om{\omega }

\def\s{\sigma }

\begin{document}

\title[]{Means in complete manifolds: uniqueness and approximation}

\author[M. Arnaudon]{Marc Arnaudon} \address{Laboratoire de Math\'ematiques et
  Applications\hfill\break\indent CNRS: UMR 7348\hfill\break\indent
  Universit\'e de Poitiers, T\'el\'eport 2 - BP 30179\hfill\break\indent
  F--86962 Futuroscope Chasseneuil Cedex, France}
\email{marc.arnaudon@math.univ-poitiers.fr}
\author[L. Miclo]{Laurent Miclo} \address{Institut de Math\'ematique de Toulouse\hfill\break\indent CNRS: UMR 5219\hfill\break\indent
  118, route de Narbonne\hfill\break\indent
  F--31062 Toulouse Cedex 9, France}
\email{laurent.miclo@math.univ-toulouse.fr}

%\keywords{}

%%%%%%%%%%%%%%%%%%%%%%%%%%%%%%%%%%%%%%%%%%%%%%%%%%%%%%%%%%%%%%%%%%%%%%%%%%
%
%  Abstract, Keywords, AMS classification
%
%%%%%%%%%%%%%%%%%%%%%%%%%%%%%%%%%%%%%%%%%%%%%%%%%%%%%%%%%%%%%%%%%%%%%%%%%%

\begin{abstract}\noindent
Let $M$ be a complete Riemannian manifold,  $N\in \NN$ and $p\ge 1$. We prove that almost everywhere on  $x=(x_1,\ldots,x_N)\in M^N$ for Lebesgue measure in $M^N$,  the measure $\di \mu(x)=\f1N\sum_{k=1}^N\d_{x_k}$ has a unique $p$-mean $e_p(x)$.
As a consequence, if $X=(X_1,\ldots,X_N)$ is a $M^N$-valued random variable with absolutely continuous law, then almost surely $\mu(X(\om))$ has a unique $p$-mean. In particular if  $(X_n)_{n\ge 1}$ is an independent sample of an absolutely continuous law in $M$, then the process $e_{p,n}(\om)=e_p(X_1(\om),\ldots, X_n(\om))$ is well-defined.

Assume $M$ is compact and consider a probability measure $\nu$ in $M$.
Using partial simulated annealing, we define a continuous semimartingale which converges to the set of minimizers of the integral of distance at power~$p$ with respect to $\nu$. When the set is a singleton, it converges to the $p$-mean. 
\end{abstract}

\maketitle
%\tableofcontents

%%%%%%%%%%%%%%%%%%%%%%%%%%%%%%%%%%%%%%%%%%%%%%%%%%%%%%%%%%%%%%%%%%%%%%%%%%%%
%
%  Actual Body of the Paper
%
%%%%%%%%%%%%%%%%%%%%%%%%%%%%%%%%%%%%%%%%%%%%%%%%%%%%%%%%%%%%%%%%%%%%%%%%%%%%

%%%%%%%%%%%%%%%%%%%%%%%%%%%%%%%%%%%%%%%%%%%%%%%%%%%%%%%%%%%%%%%%%%%%%%%%%%%%

\section{Introduction}\label{Section1}\setcounter{equation}0

Finding the mean of the median or more generaly the $p$-mean~$e_p$  of a probability measure in a manifold (the point which minimizes integral with respect to this measure of distance  at power $p$) has numerous applications. There is not much to say for the mean in $\RR^d$, almost the only case where there is a closed formula, and the most important case as the most useful estimator in statistics when the measure is uniform law on a sample. For medians in $\RR^d$ the situation is more complicated. Uniqueness holds as soon as the support of the probability measure is not carried by a line. The first algorithm for computing $e_1$ is due to Weisfeld in~\cite{Weiszfeld:37}. As for the computation of~$e_\infty$ (the center of the smallest ball containing the support of the measure), Bad\u oiu and Clarkson gave a fast and simple algorithm in~\cite{Badoiu-Clarkson:03}. For many applications in biology, signal processing, information geometry, extension to other spaces is necessary. The median in Hilbert space is computed in~\cite{Cardot-Cenac-Zitt:12}. In nonlinear spaces with convexity assumptions, uniqueness has been established in~\cite{Kendall:90} for the mean,  \cite{Afsari:10} for the $p$-mean. Many algorithms of computation now exist. As far as deterministic algorithms are concerned, one can cite~\cite{Le:04}, \cite{Groisser:05}, \cite{Groisser:06}, \cite{Afsari-Tron-Vidal:11} for the mean in Riemannian manifolds,~\cite{Arnaudon-Nielsen:12a} for the mean in Finsler manifolds,~\cite{Fletcher-al:09} and more generally~\cite{Yang:10} for the median,~\cite{Arnaudon-Nielsen:12} for~$e_\infty$. Stochastic algorithms avoid to compute the gradient of the functional to minimize. They can be found in~\cite{Sturm:05},~\cite{Arnaudon-al:12}. For other functionals to minimize, see~\cite{Bonnabel:11}.

In this paper we investigate the case of non necessarily convex, complete Riemannian manifolds. Our first result (Theorem~\ref{T1}) concerns uniqueness of the $p$-mean of the uniform measure on a finite set $\{x_1,\ldots,x_n\}$ of points, almost everywhere on $x=(x_1,\ldots,x_n)$ for the Lebesgue measure. This generalizes Bhattacharya and Patangreanu result on the circle (\cite{BP03}, case $p=2$). See also~\cite{Charlier:11} for more general uniqueness criterions on the circle.

For computation of the $p$-mean, usual deterministic algorithms are not possible any more, due to the fact that the functional to minimize may have many local minima. So restricting to symmetric spaces  we use a simulated annealing method with a continuous stochastic process, together with an estimation of the gradient to minimize via a drift moving faster and faster. With this method we are able to define a process which converges in distribution to the $p$-mean for $p\in[1,\infty)$ (Theorem~\ref{T3}, and Theorem~\ref{T2} for more general but smooth functionals). 

The main applications are in signal processing with polarimetric signal, but also for the group of rotations of $\RR^n$, so as to determine averages on rotations. Also this solves many problems of  optimization which may arise in economy, decision support, operation research.
Notice that on the circle, fast computation of the mean has been performed in~\cite{Hotz-Huckemann:11}. In fact this is a case where a closed formula can be found. For general case the situation is much more complicated and the convergence of our processes is slower and weaker. Jump processes and algorithms related to the continuous processes presented here will be investigated in a forthcoming paper.

\section{Uniqueness of $p$-means for uniform measures with finite support}\label{Section2}\setcounter{equation}0

Let $M$ be a $d$-dimensional complete Riemannian manifold with Riemannian distance denoted by $\rho$.
 For $\nu$ a probability measure on $M$ and $p\ge 1$, we define 
 \begin{equation}
\label{E1}
\begin{split}
H_{p,\nu} : M&\to \RR_+\cup\{+\infty\},\\
y& \mapsto \int_M\rho^p(y,z)\,\nu(dz).
\end{split}
\end{equation}
Either $H_{p,\nu}\equiv \infty$ or for all $y\in M$, $H_{p,\nu}(y)<\infty$. In the latter case we denote by $Q_{p,\nu}$ the set of minimizers of  $H_{p,\nu}$.  When $Q_{p,\nu}$ has only one element we denote it by $e_{p,\nu}$ and call it the $p$-mean of $\nu$. When there is no possible confusion we let $e_p=e_{p,\nu}$. 
For $x=(x_1,\ldots,x_N)\in M^N$, we let 
\begin{equation}
\label{E1bis}
\mu(x)=\f1N\sum_{k=1}^N\d_{x_k}.
\end{equation}
Clearly $H_{p,\mu(x)}$ is finite.

\begin{thm}
\label{T1}
Assume $p>1$ or  $\{$$d>1$ and $N>2$$\}$.
For almost all $x\in M^N$, $Q_{p,\mu(x)}$ has a unique element  $e_{p,\mu_(x)}$
\end{thm}
\begin{remark}
\label{R1}
This theorem extends Theorem~4.15 in \cite{Yang:11} where the same result has been established for $p=1$ and $M$ compact. 
\end{remark}

\begin{proof}

We begin with the case $p>1$.

Since $\mu(x)$ has a finite support, we can assume that $M$ is a compact Riemannian manifold. For this a smooth modification outside a large ball is sufficient. For instance we can choose a radius so that the boundary is smooth, double the ball and finally smoothen the metric locally around the place where the pasting has been performed.

So in the sequel we will assume that $M$ is compact, with diameter $L$.
For $y\in M$ we denote by $S_yM\subset T_yM$ the set of unit tangent vectors above $y$. Let
\begin{equation}
\label{E3}
{\tilde V}=\left\{(y,n), \ y\in M, \ n=(n_1,\ldots, n_N), \ n_j\in S_yM, \ j=1,\ldots N\right\}\times [0,2L]^N.
\end{equation}
Note ${\tilde V}$ is a compact  smooth $(N+1)d$-dimensional  manifold with boundary.

 Define 
 \begin{equation}
 \label{E4}
 \begin{split}
 {\tilde \phi} : {\tilde V}&\to M^N\\
 (y,n,r)&\mapsto \left(\exp_y(n_1r_1),\ldots, \exp_y(n_Nr_N\right).
 \end{split}
 \end{equation}
  The map ${\tilde \phi}$ is onto. If $x=(x_1,\ldots,x_N)\in M^N$, consider $y\in M$ minimizing $H_{p,\mu(x)}$.  Then among all $(n,r)$ such that
\begin{equation}
\label{E6}
{\tilde \phi}(y,n,r)=x
\end{equation}
we can choose one so that
for all $k=1,\ldots,N$ the map $s\mapsto \exp_y(sn_k)$ is a minimal geodesic for $s\in [0,r_k]$.  
For this choice we have 
\begin{equation}
\label{E6bis}
 H_{p,\mu(x)}(y)=\f1N\sum_{k=1}^N r_k^p.
 \end{equation}
 Now since $y$ minimizes $H_{p,\mu(x)}$, from equation~\eqref{E6bis} at $y$ and variation of arc length formula, we have for all $u\in T_yM$
 \begin{equation}
 \label{E6quad}
  \left\langle\sum_{k=1}^Nr_k^{p-1}n_k,u\right\rangle\le 0
\end{equation}
and this implies
 \begin{equation}
 \label{E6ter}
  \sum_{k=1}^Nr_k^{p-1}n_k=0.
\end{equation}

So letting 
\begin{equation}
\label{E3bis}
\tilde W_p=\left\{(y,n,r)\in \tilde V,\ \sum_{k=1}^N r_k^{p-1}n_k=0\right\}
\end{equation}
and $\tilde\phi_p=\phi|_{\tilde W_p}$ the restriction of $\tilde \phi$ to $\tilde W_p$, $\tilde \phi_p$ is onto, on $M^N$ by~\eqref{E6} and~\eqref{E6ter}.

By Sard's theorem, the set $C_1\subset M^N$ of singular values of $\tilde \phi_p$ has measure $0$. It is closed since $\tilde W_p$ is compact.
 
Let us prove that the set
\begin{equation}
\label{E7}
C_2:=\left\{(x_1,\ldots,x_N)\in M^N,\ \{x_1,\ldots,x_N\}\cap Q_{p,\mu(x_1,\ldots,x_N)}\not=\emptyset\right\}
\end{equation}
has Lebesgue-measure~$0$: 
we can assume that for $i\not=j$, $x_i\not=x_j$ since we exclude $0$-measure sets.  So the elements we consider are images by $\tilde \phi_p$ of 
\begin{equation}
\label{E7bis}
\hat W_p=\left\{(y,n,r)\in \tilde W_p,\  r_1=0,\  \forall k\ge 2 \ r_k> 0\right\}.
\end{equation}
The set~$\hat W_p$ is a submanifold of codimension $1$ of $\tilde W_p$. Now $\dim {\tilde W_p}=Nd=\dim M^N$ so  $\dim {\hat W}_p=\dim M^N-1$ and its image by $\tilde \phi_p$    is of measure $0$ in $M^N$. As a conclusion, $C_2$ has measure~$0$.

Define 
\begin{equation}
 \label{E8}
C_3:=\left\{(x_1,\ldots,x_N)\in M^N,\  \exists i\not=j\  \hbox{s.t.}\  x_i=x_j\right\}
\end{equation}
and $C=C_1\cup C_2\cup C_3$. The set $C$ is closed in $M^N$ and has measure~$0$. Letting 
\begin{equation}
\label{E39}
W_p=\left\{(y,n,r)\in \tilde W_p, \ \forall k=1,\ldots N, r_k\in (0,2L)\right\},
\end{equation}
we proved that ${\tilde \phi}_p|_{W_p}$ is onto on $M^N\backslash C$. Denote $\phi_p={\tilde \phi_p}|_{W_p}$.
 Since   $W_p$ has same dimension as $M^N$ and $\tilde W_p$ is compact,  every point $x$ of $M^N\backslash C$ has a neighbourhood $V_x$ such that $\phi_p^{-1}(V_x)=U_{1,x}\cup\cdots\cup U_{m_x,x}$ where the $U_{j,x}$ are disjoint open subsets of $W_p$ and 
\begin{equation}
\label{E6.3}
\phi_p|_{U_{j,x}} : U_{j,x}\to \phi_p(U_{j,x})
\end{equation}
is a diffeomorphism. Now since $M^N\backslash C$ is second countable we can cover it by a countable number of such sets $V_x$. So to prove that the $p$-mean is almost everywhere unique it is sufficient to prove it on $V_x$. 

%Let  $\s : W_p\to M$ the projection into the first coordinate and for $i\in\{1,\ldots,m_x\}$ 
%$$
%\s_i=\s\circ \left(\phi_p|_{U_{i,x}}\right)^{-1}.
%$$ 
For $x'\in V_x$ denote $x'=(x_1',\ldots,x_N')$, and for $i\in\{1\ldots m_x\}$, write $$(\phi|_{U_{i,x}})^{-1}(x')=(y_i(x'), n_1^i(x'),\ldots n_d^i(x'),r_1^i(x'),\ldots, r_d^i(x')).$$
Let $i,j\in\{1\ldots m_x\}$ satisfy $i\not=j$. If   $y_i(x'), y_j(x')\in Q_{p,\mu(x')}$ then we have 
\begin{equation}
\label{E10ter}
H_{p,\mu(x')}\circ y_i(x')=H_{p,\mu(x')}\circ y_j(x').
\end{equation}
We can assume with the same argument as for~\eqref{E6} and~\eqref{E6bis}  that the maps 
\begin{equation}
\label{E6.4}
\g_{i,k,x'} : s\mapsto \exp_{y_i(x')}(sn_k^i(x'))\quad\hbox{and}\quad \g_{j,k,x'} :  s\mapsto \exp_{y_j(x')}(sn_k^j(x'))
\end{equation}
 are minimal geodesics respectively on $[0,r_k^i(x')]$ and $[0,r_k^j(x')]$.
  So letting $h_p : W_p\to \RR$, $(y,n,r)\mapsto\sum_{k=1}^Nr_k^{p}$, we have 
$$
\f1N h_p\circ (\phi_p|_{U_{i,x}})^{-1}(x')=H_{p,\mu(x')}\circ y_i(x'), \quad 
\f1N h_p\circ (\phi_p|_{U_{j,x}})^{-1}(x')=H_{p,\mu(x')}\circ y_j(x').
$$
 It is sufficient to prove that for all $x'\in V_x$,  
 \begin{equation}
 \label{E10bis}
  h_p\circ (\phi_p|_{U_{i,x}})^{-1}(x')= h_p\circ (\phi_p|_{U_{j,x}})^{-1}(x') 
  \end{equation}
   implies 
\begin{equation}
 \label{E10}
\grad_{x'} \left(h_p\circ (\phi_p|_{U_{i,x}})^{-1}\right)\not=\grad_{x'} \left(h_p\circ (\phi_p|_{U_{j,x}})^{-1}\right).
\end{equation}
Indeed with~\eqref{E10} we will be able to deduce that the set 
\begin{equation}
\label{E11bis}
\left\{(x'\in V_x,\ h_p\circ (\phi_p|_{U_{i,x}})^{-1}=h_p\circ (\phi_p|_{U_{j,x}})^{-1}\right\}
\end{equation}
has codimension~$\ge 1$ in $V_x$ and this will imply that
\begin{equation}
\label{E11}
\left\{(x'\in V_x,\ H_{p,\mu(x')}\circ y_i(x')=H_{p,\mu(x')}\circ y_j(x')\right\}
\end{equation}
has codimension~$\ge 1$ in $V_x$.

Let us prove~\eqref{E10}.
For $k=1,\ldots, N$ let $$m_k^i(x')=-\dot\g_{i,k,x'}(r_k^i(x'))\quad\hbox{and}\quad m_k^j(x')=-\dot\g_{j,k,x'}(r_k^j(x')).$$ These unit vectors satisfy 
$$
\exp_{x_k'}(r_k^i(x') m_k^i(x'))=y_i(x')
\quad\hbox{and}\quad
\exp_{x_k'}(r_k^j(x') m_k^j(x'))=y_j(x').
$$
Then noting that $\left(h_p\circ (\phi_p|_{U_{i,x}})^{-1}\right)(x')=\sum_{k=1}^N(r_k^i)^p(x_k')$ we get
\begin{align*}
&d_{x'}\left(h_p\circ (\phi_p|_{U_{i,x}})^{-1}\right)(\cdot)\\&=\left\langle-p\sum_{k=1}^N(r_k^i)^{p-1}(x')n_k^i(x'),T_{x'}y_i(\cdot)\right\rangle_{T_{y_i(x')}M}\\&-p\left\langle\left((r_1^i(x'))^{p-1}m_1^i(x'),\ldots, (r_N^i(x'))^{p-1}m_N^i(x')\right),\cdot\right\rangle_{T_{x'}M^N}.
\end{align*}
Due to the fact that $(y_i(x'), n^i(x'), r^i(x'))\in W_p$, the first term in the right vanishes. So 
\begin{equation}
 \label{E13}
\grad_{x'}\left(h_p\circ (\phi_p|_{U_{i,x}})^{-1}\right)=-p\left((r_1^i(x'))^{p-1}m_1^i(x'),\ldots, (r_N^i(x'))^{p-1}m_N^i(x')\right)
\end{equation}
and similarly
\begin{equation}\label{E13bis}
\grad_{x'}\left(h_p\circ (\phi_p|_{U_{j,x}})^{-1}\right)=-p\left((r_1^j(x'))^{p-1}m_1^j(x'),\ldots, (r_N^j(x'))^{p-1}m_N^j(x')\right).
\end{equation}
Since  $y_i(x')\not=y_j(x')$ we have $(r_1^i(x'),m_1^i(x'))\not=(r_1^j(x'),m_1^j(x'))$, so $(r_1^i(x'))^{p-1}m_1^i(x')\not=(r_1^j(x'))^{p-1}m_1^j(x')$, from which we conclude that 
$$
\grad_{x'}\left(h_p\circ (\phi_p|_{U_{i,x}})^{-1}\right)\not=\grad_{x'}\left(h_p\circ (\phi_p|_{U_{j,x}})^{-1}\right).
$$
This achieves the proof for the case $p>1$.

Let us now consider the case $p=1$. The result is due to Yang in \cite{Yang:11}, we give the proof here for completeness.

The main difference is that the subset of $M^N$ of points $x=(x_1,\ldots,x_N)$ so that $x_i\in Q_{1,\mu(x)}$ for some~$i$ has positive measure. 

First consider the open subset $U$ of $M^N$ of points $x$ such that for all $i=1,\ldots,N$, $x_i\not\in Q_{1,\mu(x)}$. 

Consider the closed subset $C_0$ of $M^N$ of points $(x_1,\ldots,x_N)={\tilde \phi}(y,n,r)$, with $(y,n,r)\in {\tilde V}$ such that for all $j,k=1,\ldots N$, $n_j=\pm n_k$.  Since $d>1$ and $N>2$ this subset has Lebesgue measure~$0$. 

Replacing $M^N$ by $U$ and $C$ by $C_0\cup C$, the argument is similar until~\eqref{E10bis}. But now we will be able to prove that~\eqref{E10bis} implies~\eqref{E10} only in some neighbourhoods $V_{x,x'}$ to be precised later, of $x'\in V_x$ such that the geodesics 
$$s\mapsto \exp_{y_i(x')}(sn_k^i(x'))\quad\hbox{and}\quad s\mapsto \exp_{y_j(x')}(sn_k^j(x'))$$ are minimal respectively on $[0,r_k^i(x')]$ and $[0,r_k^j(x')]$.
But this will be sufficient since every compact subset of $V_x$ can be covered by a finite number of these neighbourhoods $V_{x,x'}$.

Making the above assumption on $x'$, the proof is similar until~\eqref{E13} and \eqref{E13bis}. Then we have \begin{equation}
 \label{E14}
\grad_{x'}\left(h_1\circ (\phi_1|_{U_{i,x}})^{-1}\right)=-\left(m_1^i(x'),\ldots, m_N^i(x')\right)
\end{equation}
and
\begin{equation}\label{E14bis}
\grad_{x'}\left(h_1\circ (\phi_1|_{U_{j,x}})^{-1}\right)=-\left(m_1^j(x'),\ldots, m_N^j(x')\right).
\end{equation}
Assume 
$$
\grad_{x'}\left(h_1\circ (\phi_1|_{U_{i,x}})^{-1}\right)=\grad_{x'}\left(h_1\circ (\phi_1|_{U_{j,x}})^{-1}\right).
$$
Then for all $k=1,\ldots, N$, $m_k^i(x')=m_k^j(x')$. In particular for $k=1$ this implies (possibly by exchanging $i$ and $j$) that $y_i(x')$ lies in the minimizing geodesic from $x_1'$ to $y_j(x')$. Now since $x'\not\in C_0$ there exists $k\in \{1,\ldots N\}$ such that $x_k'\not\in \{\exp_{y_i(x')}(sn_1^i(x')), \ s\in [-2L,2L]\}$. On the other hand since $m_k^i(x')=m_k^j(x')$,  $y_j(x')$ (or $y_i(x')$) lies on the minimizing geodesic from $x_k'$ to $y_i(x')$ (or $y_j(x')$). As a consequence there are two minimizing geodesics from $y_i(x')$ to $y_j(x')$. But this is impossible since the geodesic from $x_1'$ to $y_j(x')$ is minimizing, contains $y_i(x')$ and $x_1'\not=y_i(x')$ by the fact that we have supposed that $x_1'\not\in Q_{1,\mu(x')}$ and $y_i(x')\in Q_{1,\mu(x')}$. 
 So  
$$
\grad_{x'}\left(h_1\circ (\phi_1|_{U_{i,x}})^{-1}\right)\not=\grad_{x'}\left(h_1\circ (\phi_1|_{U_{j,x}})^{-1}\right),
$$
and by continuity this is true in a neighbourhood $V_{x,x'}$ of $x'$.

Now we consider the case where $x_1' \in Q_{1,\mu(x')}$ and $x_2' \not\in Q_{1,\mu(x')}$. We follow the same lines as in the previous part with the difference that now $y_i(x')=x_1'$ and for the definition of $U_{i,x}$ $W_1$ is replaced by 
$$
W_1^i= \{(y,n,r)\in V,\ r_1=0\}.
$$
The definition of $U_{j,x}$ remains unchanged.
  By~\cite{Yang:10} Theorem~1
$$
\left\|\f1N\sum_{k=2}^Nn_k^i(x')\right\|\le \mu_N(x')(\{x_1'\})
$$
which gives 
\begin{equation}
 \label{E15}
\left\|\sum_{k=2}^Nn_k^i(x')\right\|\le 1.
\end{equation}
Since $d>1$ and $N>2$, the submanifolds of $V_x$  images of $$\left\{(y,n,r)\in U_{i,x},\ \left\|\sum_{k=2}^Nn_k\right\|=1\right\}$$
and 
$$\left\{(y,n,r)\in U_{i,x},\ \sum_{k=2}^Nn_k=0\right\}$$
 by $\phi_1$ have measure $0$, so we can exclude them. On the subset   $$\left\{(y,n,r)\in U_{i,x},\ 0<\left\|\sum_{k=2}^Nn_k\right\|<1\right\},$$ the function $h_1$ is smooth and on its image by $\phi_1$,
\begin{equation}
 \label{E16}
\grad_{x'}\left(h_1\circ (\phi_1|_{U_{i,x}})^{-1}\right)=-\left(0,m_2^i(x'),\ldots, m_N^i(x')\right).
\end{equation}
Again
\begin{equation}\label{E16bis}
\grad_{x'}\left(h_1\circ (\phi_1|_{U_{j,x}})^{-1}\right)=-\left(m_1^j(x'),\ldots, m_N^j(x')\right).
\end{equation}
They are not equal, and this achieves the proof for this case by the same argument as before.

Finally we consider the case where $x_1', x_2' \in Q_{1,\mu(x')}$ with $x_1'=y_i(x')$ and $x_2'=y_j(x')$. We follow the same line as in the previous case, but now for the definition of $U_{j,x}$, $W_1$ is replaced by 
$$
W_1^j= \{(y,n,r)\in V,\ r_2=0\}.
$$
 Again we can exclude the submanifolds of $V_x$  images of $$\left\{(y,n,r)\in U_{j,x},\ \left\|\sum_{k\in\{1,\ldots,N\}, k\not=2}n_k\right\|=1\right\}$$
 and
 $$\left\{(y,n,r)\in U_{j,x},\ \sum_{k\in\{1,\ldots,N\}, k\not=2}n_k=0\right\}$$
  by $\phi_1$ and work on 
\begin{align*}&\phi_1\left(\left\{(y,n,r)\in U_{j,x},\ 0<\left\|\sum_{k\in\{1,\ldots,N\}, k\not=2}n_k\right\|<1\right\}\right)\\& \qquad\cap \phi_1\left(\left\{(y,n,r)\in U_{i,x},\ 0<\left\|\sum_{k=2}^Nn_k\right\|<1\right\}\right).
\end{align*}
On this set $h_1\circ (\phi_1|_{U_{i,x}})^{-1}$ and $h_1\circ (\phi_1|_{U_{j,x}})^{-1}$ are smooth and 
\begin{equation}
 \label{E17}
\grad_{x'}\left(h_1\circ (\phi_1|_{U_{i,x}})^{-1}\right)=-\left(0,m_2^i(x'),\ldots, m_N^i(x')\right).
\end{equation}
\begin{equation}
 \label{E18}
\grad_{x'}\left(h_1\circ (\phi_1|_{U_{j,x}})^{-1}\right)=-\left(m_1^j(x'),0,m_3^j(x'),\ldots, m_N^j(x')\right).
\end{equation}
They are not equal, and this achieves the proof.
\end{proof}

\begin{cor}
\label{C1}
Let $p\in [1,\infty)$ and $X=(X_1,\ldots, X_N)$ a random variable with values in $M^N$, which has an absolutely continuous law. Then almost-surely $\mu(X(\om))$ has a unique $p$-mean $e_p(X(\om))$. 
\end{cor}
\begin{cor}
\label{C2}
Let $p\in [1,\infty)$ and $(X_n)_{n\ge 1}$ a sequence of i.i.d. $M$-valued random variables with absolutely continuous laws. Then the process of empirical $p$-means$$\Bigl(e_{p,n}(\om):=e_p\bigl(X_1(\om),\ldots,X_n(\om)\bigr)\Bigr)_{n\ge 1}$$
is well-defined.
\end{cor}
\begin{remark}
For $p=2$ and $M$ a circle, it has been proved in~\cite{BP03} that the assumption can be weakened: the same result holds if the law has no atom. 
\end{remark}
We believe that it would be interesting to study the behaviour of the process $(e_{p,n})_{n\ge 1}$ in many situations. For instance when the law of $X_1$ is uniform on a compact symmetric space (even the case of the circle is highly non trivial) one would observe a recurrent but irregular and slower and slower process. Again on a compact symmetric space, when the law $\nu$ of $X_1$ has a finite number of $p$-means due to a finite group of symmetries, one would observe an almost stationary behaviour, and at increasingly spaced times jumps between smaller and smaller neighbourhoods of the $p$-means of $\nu$.

\section{Finding the minimizers of some integrated functionals with simulated annealing}\label{Section3}\setcounter{equation}0

Let $M$ be a compact Riemannian manifold. For simplicity and without loss of generality we assume that $M$ has Lebesgue volume~$1$. On $M$ consider a probability law $\nu$ with a density with respect to Lebesgue measure, also denoted by $\nu$.
Assume we are given a continuous function
$\kappa : M\times M\to \RR_+$, where $\kappa(\theta, y)$ is  interpreted as some kind of cost for going from~$\theta$ to~$y$.
Assume furthermore that for all $y\in M$ the function $\theta\mapsto \kappa(\theta, y)$ is smooth and that its first and second derivative in $\theta$ are uniformly bounded in $(\theta,y)$.
 Consider on $M$ the functional 
\begin{equation}
\label{SA-1}
\begin{split}
U : M&\to \RR_+\\
 \theta&\mapsto \int_M\kappa(\theta,y)\nu(dy)
\end{split}
\end{equation} Denote by ${\mathcal M}$ the set of minimizers of $U$.  The aim of this section is to find a continuous semimartingale which converges in law to ${\mathcal M}$. Also we try to avoid using the gradient of $U$, which in many cases is difficult or impossible to compute. 

For this we will use a sequence $(P_k)_{k\ge0}$ of independent random variables with law $\nu$, a Poisson process $N_t$ on $\NN$ with intensity $\g_t^{-1}$ where
\begin{equation}
\label{SA0-1}
\g_t=(1+t)^{-1}.
\end{equation}
 Define
 \begin{equation}
 \label{SA0ter}
 c(U)=2\sup_{\theta,y\in M}\left(\inf_{\phi\in \SC_{\theta,y}}e(\phi)\right),
 \end{equation}
 $\SC_{\theta,y}$ denoting the set of continuous paths $[0,1]\to M$ and for $\phi\in \SC_{\theta,y}$, the elevation $e(\phi)$ being defined as 
 \begin{equation}
 \label{SA0quad}
 e(\phi)=\sup_{0\le t\le 1}U(\phi(t))-U(\theta)-U(y)+\inf_{z\in M}U(z).
 \end{equation}
 Let 
 \begin{equation}
\label{SA0}\b_t=\f1k\ln(1+t),
\end{equation}
 the constant $k$ satisfying
 $k>c(U)$.
 
  We assume that $(N_t)_{t\ge 0}$ is independent of the sequence $(P_k)_{k\ge0}$. We let $(B_t)_{t\ge 0}$ be a Brownian motion with values in $\RR^r$ for some $r\in \NN$, independent of $(N_t)_{t\ge 0}$ and  $(P_k)_{k\ge0}$, and  $\s$  a smooth section  of $TM\otimes (\RR^r)^\ast$: for all $\theta\in M$, $\s(\theta)$ is a linear map $\RR^r\to T_\theta M$. We assume that for all $\theta\in M$, we have $\s(\theta)\s(\theta)^\ast=\id_{T_\theta M}$.
 We fix $\theta_0\in M$ and let $\Theta_t$ be the solution started at $\theta_0$  of the It\^o equation 
\begin{equation}
\label{SA1}
d\Theta_t=\s(\Theta_t)\,dB_t-\b_t\grad_{\Theta_t}\kappa(\cdot,Y_t)\,dt\quad\hbox{with}\quad Y_t=P_{N_t}.
\end{equation}
Recall that if $P(\Theta_t) : T_{\theta_0}M\to T_{\Theta_t}M$ is the parallel transport map along $(\Theta_t)$, then 
\begin{equation}
\label{SA2}
d\Theta_t=P(\Theta_t)d\left(\int_0^\cdot P(\Theta_s)^{-1}\,\circ d\Theta_s\right)_t.
\end{equation}
 Also define $\Theta^{0}_t$ the solution started at $\theta_0$ of the It\^o equation
 \begin{equation}
 \label{SA4} d\Theta_t^{0}=\s(\Theta_t^{0})\,dB_t-\b_t\left(\int_M\grad_{\Theta_t^0}\kappa(\cdot,y)\,\nu(y)dy\right)\,dt.
 \end{equation}
 Note \eqref{SA4} rewrites as 
 \begin{equation}
 \label{SA5} d\Theta_t^{0}=\s(\Theta_t^{0})\,dB_t-\b_t\grad_{\Theta_t^{0}}U\,dt,
 \end{equation}
 so that the same equation with fixed $\beta$ instead of $\beta_t$ has an invariant law with density
 \begin{equation}
 \label{SA6}
 \mu_\b(\theta)=\f1{Z_\b}e^{-2\b U(\theta)},\quad\hbox{with}\quad Z_\b=\int_Me^{-2\b U(\theta')}\,d\theta'.
 \end{equation}
 The process $\Theta^{0}_t$ is an inhomogeneous diffusion with generator 
 \begin{equation}
 \label{SA6bis}
 L_t^{0}(\theta)=\f12\D(\theta)-\b_t\grad_{\theta}U.
 \end{equation}
 Denote by $m_t(\theta)$ the density of $\Theta_t$. 
 
 The process $(\Theta_t, Y_t)$ is Markovian with generator $L_t$ given by 
 \begin{equation}
 \label{SA7}
 \begin{split}
 L_tf(\theta,y)&=\left(\f12\D(\theta)-\b_t\grad_{\theta}\kappa(\cdot, y)\right)f(\cdot,y)
 +\g_t^{-1}\int_M\left(f(\theta,z)-f(\theta,y)\right)\,\nu(dz)\\
 &=L_{1,t}f(\cdot,y)(\theta)+L_{2,t}f(\theta,\cdot)(y).
 \end{split}
 \end{equation}
 
 We know that for all neighbourhood ${\mathcal N}$ of ${\mathcal M}$,  $\int_{\mathcal N}\mu_\b(\theta)\,d\theta$ converges to~$1$  as $\b\to\infty$. So to prove that $\int_{\mathcal N} m_t(\theta)\,d\theta$ converges to~$1$ it is sufficient to prove the following proposition:
 \begin{prop}
 \label{P1}
  The entropy
 \begin{equation}
 J_t:=\int_M\ln\left(\f{m_t(\theta)}{\mu_{\b_t}(\theta)}\right)m_t(\theta)\,d\theta
 \end{equation}
 converges to $0$ as  $t\to\infty$.
 \end{prop}
 
 \begin{proof}
 Let us compute 
 \begin{equation}
 \label{SA8.0}
 \begin{split}
 \f{dJ_t}{dt}=\int_M\f{dm_t(\theta)}{dt}\,d\theta-\int_M\f{d \ln \mu_{\b_t}(\theta)}{dt}m_t(\theta)\,d\theta+\int_M\ln\left(\f{m_t(\theta)}{\mu_{\b_t}(\theta)}\right)\f{dm_t(\theta)}{dt}\,d\theta.
 \end{split}
 \end{equation}
 Since for all $t$ $m_t(\theta)$ is a probability density, the first term in the right vanishes. So we get 
 \begin{equation}
 \label{SA8}
 \f{dJ_t}{dt}=2\b'_t\int_MU(\theta)(m_t(\theta)-\mu_{\b_t}(\theta))\,d\theta+\int_M L_t\left[\ln\left(\f{m_t(\theta)}{\mu_{\b_t}(\theta)}\right)\right]m_t(\theta)\,d\theta
 \end{equation}
 where the last term comes from Dynkin formula.
 For the first term in the right we have using~\eqref{SA0}
 \begin{equation}
 \label{SA9}
 \begin{split}
 2\b'_t\int_MU(\theta)(m_t(\theta)-\mu_{\b_t}(\theta))\,d\theta\le 4\|U\|_\infty|\b_t'|\le \f{4\|\kappa\|_\infty}{k(1+t)}.
 \end{split}
 \end{equation}
 Now we split the second term in the right of~\eqref{SA8} into 
 \begin{equation}
 \label{SA10}
 \begin{split} &\int_ML_t\left[\ln\left(\f{m_t(\theta)}{\mu_{\b_t}(\theta)}\right)\right]m_t(\theta)\,d\theta\\
 &=\int_ML_t^{0}\left[\ln\left(\f{m_t(\theta)}{\mu_{\b_t}(\theta)}\right)\right]m_t(\theta)\,d\theta
 +\int_MR_t(\theta,y)\left[\ln\left(\f{m_t(\theta)}{\mu_{\b_t}(\theta)}\right)\right]m_t(\theta)\,d\theta.
 \end{split}
 \end{equation}
 We have 
 \begin{equation}
 \label{SA11}
 \begin{split} &\int_ML_t^{0}\left[\ln\left(\f{m_t(\theta)}{\mu_{\b_t}(\theta)}\right)\right]m_t(\theta)\,d\theta\\
 &=\int_ML_t^{0}\left[\left(\f{m_t(\theta)}{\mu_{\b_t}(\theta)}\right)\right]\mu_{\b_t}(\theta)\,d\theta
 -\f12 \int_M\left\|\n \ln\left(\f{m_t(\theta)}{\mu_{\b_t}(\theta)}\right)\right\|\mu_{\b_t}(\theta)\,d\theta \\
 &=-2\int_M\left\|\n\sqrt{\f{m_t(\theta)}{\mu_{\b_t}(\theta)}}\right\|^2\mu_{\b_t}(\theta)\,d\theta\\
 &\le - 2c_2(\b_t\vee 1)^{-p}\exp\left(-c(U)\b_t\right)J_t
 \end{split}
 \end{equation}
 for some $c_2>0$ and integer $p>0$
 by logarithmic Sobolev inequality (\cite{Holley-Kusuoka-Stroock:89} and \cite{Holley-Stroock:88}, for more details see~\cite{Miclo:92}). Note we used again Dynkin formula to prove the vanishing of the first term in the right of the second line.
 
 As for the second term we have 
 \begin{align*} &\int_MR_t(\theta,y)\left[\ln\left(\f{m_t(\theta)}{\mu_{\b_t}(\theta)}\right)\right]m_t(\theta)\,d\theta\\
 &=\int_M-\b_t\left\langle d\ln\left(\f{m_t(\theta)}{\mu_{\b_t}(\theta)}\right),\grad_{\theta}\kappa(\cdot, y)-\int_M\grad_{\theta}\kappa(\cdot, z)\,\nu(dz)\right\rangle m_t(\theta,y)\,d\theta \,dy\\
 &=-\b_t\int_M\left\langle d\ln\left(\f{m_t(\theta)}{\mu_{\b_t}(\theta)}\right),\int_M \grad_{\theta}\kappa(\cdot, y)\,(m_t(y|\theta)-\nu(y))\right\rangle m_t(\theta)\,d\theta \,dy\\
 &=2\b_t\int_M\sqrt{\f{\mu_{\b_t}}{m_t}(\theta)}\left\langle d\sqrt{\f{m_t}{\mu_{\b_t}}(\theta)},R_t(\theta)\right\rangle m_t(\theta)\,d\theta
 \end{align*}
 with 
 $$
 R_t(\theta)=-\int_M\grad_{\theta}\kappa(\cdot, y)(m_t(y|\theta)-\nu(y))\,dy.
 $$
 So by Cauchy-Schwartz inequality
 \begin{align*} &\int_MR_t(\theta,y)\left[\ln\left(\f{m_t(\theta)}{\mu_{\b_t}(\theta)}\right)\right]m_t(\theta)\,d\theta\\
 &\le 2\b_t\left(\int_M\left\|\n\sqrt{\f{m_t}{\mu_{\b_t}}(\theta)}\right\|^2\mu_{\b_t}(\theta)\,d\theta\right)^{1/2}\left(\int_M\|R_t(\theta)\|^2m_t(\theta)\,d\theta\right)^{1/2}\\
 &\le \b_t^2\int_M\|R_t(\theta)\|^2m_t(\theta)\,d\theta+\int_M\left\|\n\sqrt{\f{m_t}{\mu_{\b_t}}(\theta)}\right\|^2\mu_{\b_t}(\theta)\,d\theta.
 \end{align*}
 Replacing $2c_2$ by $c_2$ in~\eqref{SA11} we can after summing get rid of the second term in the right. 
 Defining
 \begin{equation}
 \label{SA11bis}
 K=\sup_{\theta,y\in M}\|\grad_\theta\kappa(\cdot,y)\|,
 \end{equation}
  let us now bound 
 \begin{align*} \int_M\|R_t(\theta)\|^2m_t(\theta)\,d\theta&=\int_M\left\|\int_M\grad_{\theta}\kappa(\cdot, y)(m_t(y|\theta)-\nu(y))\,dy\right\|^2m_t(\theta)\,d\theta\\&
 \le\int_M\left\|K\int_M|m_t(y|\theta)-\nu(y)|\,dy\right\|^2m_t(\theta)\,d\theta
 \\&\le 32K^2\int_M\left(\int_M\ln\left(\f{m_t(y|\theta)}{\nu(y)}\right)m_t(y|\theta)\,dy\right)m_t(\theta)\,d\theta\\
 &=32K^2I_t
 \end{align*}
 where we have defined 
 \begin{equation}
 \label{SA12}
 I_t=\int_{M\times M}\ln\left(\f{m_t(y|\theta)}{\nu(y)}\right)m_t(y,\theta)\,dy.
 \end{equation}
 We also used classical bound of total variation by entropy (\cite{Holley-Stroock:88}):
 $$
 \int_M|m_t(y|\theta)-\nu(y)|\,dy\le 4\sqrt2\left(\int_M\ln\left(\f{m_t(y|\theta)}{\nu(y)}\right)m_t(y|\theta)\,dy\right)^{1/2}
 $$
 At this stage we proved that 
 \begin{equation}
 \label{SA13}
 \f{dJ_t}{dt}\le \f{4\|\kappa\|_\infty}{k(1+t)}-c_2(\b_t\vee 1)^{-p}\exp\left(-c(U)\b_t\right)J_t+\b_t^2 32K^2I_t.
 \end{equation}
 
 The next step is to find a suitable bound for $\f{dI_t}{dt}$. As before 
 \begin{equation}
 \label{SA14}
 \begin{split}
 \f{dI_t}{dt}&=\int_{M\times M}L_t\left[\ln\left(\f{m_t(y|\theta)}{\nu(y)}\right)\right]m_t(y,\theta)\,d\theta dy\\
 &=\int_{M\times M}(L_{2,t}+L_{1,t})\left[\ln\left(\f{m_t(y|\theta)}{\nu(y)}\right)\right]m_t(y,\theta)\,d\theta dy.
 \end{split}
 \end{equation}
 We begin with the first term:
 \begin{align*}
 &\int_{M\times M}L_{2,t}\left[\ln\left(\f{m_t(y|\theta)}{\nu(y)}\right)\right]m_t(y,\theta)\,d\theta dy\\
 &=\g_t^{-1}\int_{M\times M}\int_M\left[\ln\left(\f{m_t(z|\theta)}{\nu(z)}\right)-\ln\left(\f{m_t(y|\theta)}{\nu(y)}\right)\right]\nu(dz)m_t(y,\theta)\, d\theta dy\\
 &=\g_t^{-1}\int_{M\times M}\ln\left(\f{m_t(y|\theta)}{\nu(y)}\right)\left(\nu(y)-m_t(y|\theta)\right)m_t(\theta)\,d\theta
 %\\
 %&\le -\f{2\b_t}{\g_t}\int_{M\times M}\left(\sqrt{\nu(y)}-\sqrt{m_t(y|\theta)}\right)^2m_t(\theta)\,dyd\theta
 .\end{align*}
% where we used the inequality
 %$$
 %(a-b)(\ln a-\ln b)\ge 2(\sqrt{a}-\sqrt{b})^2.
 %$$
 By Jensen inequality we have 
 \begin{align*}
 &\int_{M\times M}\ln\left(\f{m_t(y|\theta)}{\nu(y)}\right)\nu(y)m_t(\theta)\,dy d\theta\\
 &=\int_M\left(\int_M\ln\left(\f{m_t(y|\theta)}{\nu(y)}\right)\nu(y)\,dy\right)m_t(\theta)\,d\theta\\
 &\le\int_M\ln\left(\int_M\f{m_t(y|\theta)}{\nu(y)}\nu(y)\,dy\right)m_t(\theta)\,d\theta\\
 &=\int_M\ln(1)m_t(\theta)\,d\theta=0
 \end{align*}
  Consequently
  \begin{align*}
  &\int_{M\times M}L_{2,t}\left[\ln\left(\f{m_t(y|\theta)}{\nu(y)}\right)\right]m_t(y,\theta)\,d\theta dy\\&\le
  -\g_t^{-1}\int_{M\times M}\ln\left(\f{m_t(y|\theta)}{\nu(y)}\right)m_t(y|\theta)m_t(\theta)\,d\theta
  \end{align*}
which rewrites as 
\begin{equation}
\label{SA15}
\int_{M\times M}L_{2,t}\left[\ln\left(\f{m_t(y|\theta)}{\nu(y)}\right)\right]m_t(y,\theta)\,d\theta dy\le -\g_t^{-1}I_t.
\end{equation}
Let us now consider the second term in the right of~\eqref{SA14}. Since 
$$
\ln\left(\f{m_t(y|\theta)}{\nu(y)}\right)=\ln\left(\f{m_t(\theta|y)}{m_t(\theta)}\right)
$$
it rewrites as
\begin{equation}
\label{SA16}
\begin{split}
&\int_{M\times M}L_{1,t}\left[\ln\left(\f{m_t(\theta|y)}{m_t(\theta)}\right)\right]
m_t(y,\theta)\,d\theta dy\\
&=\int_{M\times M}L_{1,t}\left[\ln (m_t(\theta|y))-\ln (m_t(\theta))\right]
m_t(y,\theta)\,d\theta dy.
\end{split}
\end{equation}
But 
\begin{equation}
\label{SA17}
\begin{split}
&\int_{M\times M}L_{1,t}\ln (m_t(\theta|y))
m_t(y,\theta)\,d\theta dy\\
&=\f12 \int_{M\times M}\D\ln (m_t(\theta|y))-\b_t\int_M\left\langle d_\theta \ln m_t(\cdot|y),\grad_{\theta}\kappa(\cdot, y)\right\rangle m_t(y,\theta)\,d\theta dy\\
&=-2\int_{M\times M}\left\|\n\sqrt{m_t(\theta|y)}\right\|^2\,d\theta \nu(dy)\\&-\b_t\int_{M\times M}\left\langle d_\theta m_t(\cdot|y),\grad_{\theta}\kappa(\cdot, y)\right\rangle 
\nu(y)\,d\theta dy.
\end{split}
\end{equation}
Let us bound the absolute value of the last term:
\begin{equation}
\label{SA18}
\begin{split}
&\left|-\b_t\int_{M\times M}\left\langle d_\theta m_t(\cdot|y),\grad_{\theta}\kappa(\cdot, y)\right\rangle \nu(y)\,d\theta dy
\right|\\
&=\left|2\b_t \int_{M\times M}\left\langle d_\theta\sqrt{m_t(\cdot|y)}, \grad_{\theta}\kappa(\cdot, y)\right\rangle \sqrt{m_t(\theta|y)}\nu(dy)\,d\theta dy\right|\\
&\le 2\b_t K\int_{M\times M}\left\|\n \sqrt{m_t(\theta|y)}\right\| \sqrt{m_t(\theta|y)}\nu(y) \,d\theta dy\\
&\le \int_{M\times M}\left(\f12\b_t^2 K^2m_t(\theta|y)+2\left\|\n \sqrt{m_t(\theta|y)}\right\|^2\right)\nu(y) \,d\theta dy\\
&=\f12\b_t^2K^2+2\int_{M\times M}\left\|\n \sqrt{m_t(\theta|y)}\right\|^2\nu(y) \,d\theta dy.
\end{split}
\end{equation}
This yields 
\begin{equation}
\label{SA19}
\int_{M\times M}L_{1,t}\ln (m_t(\theta|y))
m_t(y,\theta)\,d\theta dy\le \f12\b_t^2K^2
\end{equation}
We also have to bound the  last term in~\eqref{SA16}: 
\begin{equation}
\label{SA20}
\begin{split}
-L_{1,t} \ln(m_t(\theta))=-L_{1,t}\ln\left(\f{m_t(\theta)}{\mu_{\b_t}(\theta)}\right)-L_{1,t}\ln(\mu_{\b_t}(\theta)).
\end{split}
\end{equation}
We already know that
\begin{equation}
\label{SA21}
-\int_ML_{1,t}\ln\left(\f{m_t(\theta)}{\mu_{\b_t}(\theta)}\right)m_t(\theta)\,d\theta\le -\f{dJ_t}{dt}+4\|\kappa\|_\infty\b_t'.
\end{equation}

For the second term we have
\begin{equation}
\label{SA22}
\begin{split}
-L_{1,t}\ln(\mu_{\b_t}(\theta))&=2L_{1,t}U(\theta)\\
&=\b_t\D U(\theta)+\b_t^2\left\langle dU, \grad_{\theta}\kappa(\cdot, y)\right\rangle\\
&\le K'(\b_t\vee 1)\b_t
\end{split}
\end{equation}
with 
\begin{equation}
\label{SA23bis}
K'=\sup_{\theta,y\in M}|\D_\theta\kappa(\cdot,y)|+K^2.
\end{equation}

Finally we obtain 
\begin{equation}
\label{SA23}
\f{dI_t}{dt}\le 4\|\kappa\|_\infty \b_t'+K'(\b_t\vee 1)\b_t-\f{dJ_t}{dt}-\g_t^{-1}I_t
\end{equation}
together with~\eqref{SA13}:
\begin{equation}
\label{SA13bis}
 \f{dJ_t}{dt}\le \f{4\|\kappa\|_\infty}{k(1+t)}-c_2(\b_t\vee 1)^{-p}\exp\left(-c(U)\b_t\right)J_t+2\b_t^232K^2I_t.
 \end{equation}
 At this stage we can use the end of the proof of theorem 1 in \cite{Miclo:96} to obtain that under assumptions~\eqref{SA0} and~\eqref{SA0ter} then 
 \begin{equation}
 \label{SA24}
 \lim_{t\to\infty} J_t=0
 \end{equation}
 (notice that in Section~\ref{Section4} we will prove this in a more general context).
 \end{proof}
 \begin{thm}
 \label{T2} Assume 
 \begin{equation}
\label{SA0bis}\b_t=\f1k\ln(1+t),\quad\hbox{and}\quad \g_t=(1+t)^{-1},
\end{equation}
where $k>c(U)$, $(c(U)$ defined in~\eqref{SA0ter}). 
Then for any neighbourhood ${\mathcal N}$ of ${\mathcal M}$, 
\begin{equation}
\label{SA25}
\lim_{t\to\infty}\PP\left[\Theta_t\in {\mathcal N}\right]=1.
\end{equation}
 \end{thm}
 \begin{proof}
 We use Proposition~\ref{P1} together with the fact that 
 $$
 \|m_t-\mu_{\b_t}\|\le 4\sqrt{2J_t}
 $$
 and
 $$
 \lim_{t\to\infty}\mu_{\b_t}({\mathcal N})=1.
 $$
 \end{proof}
 \section{Application to location of $p$-means in symmetric spaces}\label{Section4}\setcounter{equation}0
 In this section we assume that $M$ is a compact symmetric space endowed with the canonical Riemannian metric of volume~$1$. Denote by $\rho$ the Riemannian distance in $M$, $D$ its diameter. We fix $p\ge 1$ and consider a probability measure $\nu$ on $M$. We aim to find at least one element of $Q_{p,\nu}$  by using the result of the previous section. In particular if $\nu$ has a unique $p$-mean $e_p$, then we will be able to construct a process which converges in probability to $e_p$ as $t\to \infty$. 
 
 Denote by $p(s,x,y)$ the heat kernel on $M$, and for $s>0$ let $\nu_s$ be the probability measure with density
 \begin{equation}
 \label{FEP01}
 \nu_s(y)=\int_Mp(s,y,z)\nu(dz),
 \end{equation}
 and let 
 \begin{equation}
 \label{FEP02}
 \begin{split}
 \kappa_s : M\times M&\to \RR_+\\
  (\theta,y)&\mapsto \int_Mp(s,\theta, z)\rho^p(z,y)\,dz,
  \end{split}
 \end{equation}
 and 
 \begin{equation}
 \label{FEP03}
 \begin{split}
 U_{s_1,s_2} : M&\to \RR_+\\
  \theta&\mapsto \int_M\kappa_{s_1}(\theta,y)\nu_{s_2}(y)\,dy.
  \end{split}
 \end{equation}
 Also let $U=H_{p,\nu}$.
 Clearly $\nu_{s_1}$ and $\kappa_{s_2}$ satisfy the assumption of the previous section. Moreover, denoting by ${\mathcal M}_{s_1,s_2}$ the set of minimizers of $U_{s_1,s_2}$ then as $s_1,s_2\to 0$ we have ${\mathcal M}_{s_1,s_2}\to Q_{p,\nu}$ is the sense that for any neighbourhood ${\mathcal N}$ of $Q_{p,\nu}$, we have ${\mathcal M}_{s_1,s_2}\subset {\mathcal N}$ for all $s_1,s_2$ sufficiently small. This is due to the fact that as $s_1,s_2\to 0$, $U_{s_1,s_2}(\theta)\to U(\theta)$ uniformly in $\theta$.
 \begin{lemma}
 \label{L2}
 For all $s_1,s_2>0$ we have 
 \begin{equation}
 \label{FEP04}
 U_{s_1,s_2}(\theta)=U_{0,s_1+s_2}(\theta)=\int_M\rho(\theta,y)\nu_{s_1+s_2}(y)\,dy.
 \end{equation}
 \end{lemma}
 \begin{proof}
 Fix $\theta,y\in M$, let $m$ be the middle point of a minimal geodesic from $\theta$ to $y$ and $i_m$ the symmetry centered at $m$. We have 
 \begin{align*} \int_Mp(s_1,\theta,z)\rho^p(z,y)\,dz&=\int_Mp(s_1,i_m(\theta),i_m(z))\rho^p(i_m(z),i_m(y))\,dz\\
 &=\int_Mp(s_1,i_m(\theta),z')\rho^p(z',i_m(y))\,dz'\\
 &=\int_Mp(s_1,y,z')\rho^p(z',\theta)\,dz'\\
 &=\int_M\rho^p(\theta,z')p(s_1,z',y)\,dz'
 \end{align*}
 where we first used the invariance by isometry of the heat kernel and then did the change of variable $z'=i_m(z)$ in the integral and finally used the symmetry of the heat kernel.  To finish the proof we are left to use the convolution property of the heat semigroup. 
 \end{proof}
 \begin{cor}
 \label{C3}
 Defining 
 \begin{equation}
 \label{FEP11bis}
 K''=\sup_{\theta,y\in M}\|\grad_\theta\rho(\cdot,y)\|,
 \end{equation}
 we have 
  for all $s_1, s_2>0$, $\theta,y\in M$,
 \begin{equation}
 \label{FEP05}
 \|\grad_\theta\kappa_{s_1}(\cdot,y)\|\le pD^{p-1}K''=:K \quad \hbox{and}\quad \|\grad_{\theta}U_{s_1,s_2}\| \le K.
 \end{equation}
 \end{cor}
 
 With all these properties we would like to find $s_1(t)\searrow 0$ and $s_2(t)\searrow 0$ such that the process $\Theta_t$ started at $\theta_0$ and solution to 
 \begin{equation}
 \label{FEP06}
 d\Theta_t=\s(\Theta_t)\,dB_t-\b_t\grad_{\Theta_t}\kappa_{s_1(t)}(\cdot,Y_t^{s_2(t)})\,dt
 \end{equation}
 converges in law to $e_p$, $(N_t,Y_t^{s_2(t)})$ being a Poisson point process in $[0,\infty)\times M$ with intensity $\g(t)^{-1}\nu_{s_2(t)}(y)\,dt\,dy$, independent of $(B_t)$. This is the object of the next theorem in which we will take
 $$
 s_1(t)=s_2(t)=s_t=(\ln(1+t))^{-1}.
 $$

 So define $\Theta^{0}_t$ the solution started at $\theta_0$ of the It\^o equation
 \begin{equation}
 \label{FEP4} d\Theta_t^{0}=\s(\Theta_t^{0})\,dB_t-\b_t\left(\int_M\grad_{\Theta_t^0}\kappa_{s_t}(\cdot,y)\,\nu_{s_t}(y)dy\right)\,dt.
 \end{equation}
 Notice that using Lemma~\ref{L2}, \eqref{FEP4} rewrites as 
 \begin{equation}
 \label{FEP5} d\Theta_t^{0}=\s(\Theta_t^{0})\,dB_t-\b_t\grad_{\Theta_t^{0}}U_{2s_t}\,dt,
 \end{equation}
 where $U_{2s_t}:=U_{0,2s_t}$, 
 so that the same equation with fixed $(\beta,s)$ instead of $(\beta_t,s_t)$ has an invariant law with density
 \begin{equation}
 \label{FEP6}
 \mu_{\b,s}(\theta)=\f1{Z_{\b,s}}e^{-2\b U_{2s}(\theta)},\quad\hbox{with}\quad Z_{\b,s}=\int_Me^{-2\b U_{2s}(\theta')}\,d\theta'.
 \end{equation}
 
%Also using the fact that as $s\searrow 0$, $U_s$ converges uniformly to $U:=U_0$, if we denote by ${\mathcal M}$ (resp. ${\mathcal M}_s$) the set of minimizers of~$U$) (resp. of $U_s$), then for all neighborhood ${\mathcal N}$ of ${\mathcal M}$, we have ${\mathcal M_s}\subset {\mathcal N}$ for all $s$ sufficiently large.
 The process $\Theta^{0}_t$ is an inhomogeneous diffusion with generator 
 \begin{equation}
 \label{FEP6bis}
 L_t^{0}(\theta)=\f12\D(\theta)-\b_t\grad_{\theta}U_{2s_t}.
 \end{equation}
 Denote by $m_t(\theta)$ the density of $\Theta_t$. 
  
  Let $Y_t:=Y_t^{s_t}$.
 The process $(\Theta_t, Y_t)$ is Markovian with generator $L_t$ given by 
 \begin{equation}
 \label{FEP7}
 \begin{split}
 L_tf(\theta,y)&=\left(\f12\D(\theta)-\b_t\grad_{\theta}\kappa_{s_t}(\cdot, y)\right)f(\cdot,y)
 +\g_t^{-1}\int_M\left(f(\theta,z)-f(\theta,y)\right)\,\nu_{s_t}(dz)\\
 &=L_{1,t}f(\cdot,y)(\theta)+L_{2,t}f(\theta,\cdot)(y).
 \end{split}
 \end{equation}
 
 We know that for all neighbourhood ${\mathcal N}$ of $Q_{p,\nu}$,  $\int_{\mathcal N}\mu_{\b,s}(\theta)\,d\theta$ converges to~$1$  as $\b\to\infty$, uniformly in $s$ sufficiently small. 
 Again define
 \begin{equation}
 J_t:=\int_M\ln\left(\f{m_t(\theta)}{\mu_{\b_t,s_t}(\theta)}\right)m_t(\theta)\,d\theta.
 \end{equation}
 
 \begin{thm}
 \label{T3}
 Assume 
 \begin{equation}
 \label{FEP0}
 \b_t=\f1k\ln(1+t), \  \g_t=(1+t)^{-1}, \  s_1(t)=s_2(t)=s(t)=(\ln(1+t))^{-1}.
 \end{equation}
 where $k>c(U)$, $(c(U)$ defined in~\eqref{SA0ter}). 
Then for any neighbourhood ${\mathcal N}$ of $Q_{p,\nu}$, the process $\Theta_t$ defined in equation~\eqref{FEP06} satisfies
\begin{equation}
\label{FEP08}
\lim_{t\to\infty}\PP\left[\Theta_t\in {\mathcal N}\right]=1.
\end{equation}
 \end{thm}
 \begin{proof}
 We use Proposition~\ref{P2} below together with the fact that 
 $$
 \|m_t-\mu_{\b_t,s_t}\|\le 4\sqrt{2J_t}
 $$
 and
 $$
 \lim_{t\to\infty}\mu_{\b_t,s_t}({\mathcal N})=1.
 $$
 \end{proof}

 \begin{prop}
 \label{P2}
  The entropy
 \begin{equation}
 J_t=\int_M\ln\left(\f{m_t(\theta)}{\mu_{\b_t,s_t}(\theta)}\right)m_t(\theta)\,d\theta
 \end{equation}
 converges to $0$ as  $t\to\infty$.
 \end{prop}
 
 \begin{proof}
 Let us compute as before
 \begin{equation}
 \label{FEP8}
 \f{dJ_t}{dt}=-\int_M\partial_t\ln(\mu_{\b_t,s_t}(\theta))m_t(\theta))\,d\theta+\int_M L_t\left[\ln\left(\f{m_t(\theta)}{\mu_{\b_t}(\theta)}\right)\right]m_t(\theta)\,d\theta.
 \end{equation}
 For the first term in the right we have using~\eqref{FEP6}
 \begin{equation}
 \label{FEP9}
 \begin{split}
 &\partial_t\ln(\mu_{\b_t,s_t}(\theta))\\&=-2\b_t'U_{2s_t}-2\b_t\int_{M\times M}2s_t'\partial_s\ln p(2s_t,\theta,z)p(2s_t,\theta,z)\rho^p(z,y)\,\nu(dy)dz\\
 &+2\b_t'\int_M U_{2s_t}(\theta')\mu_{\b_t,s_t}(\theta')\,d\theta'\\&
 +2\b_t\int_M\left(\int_{M\times M}2s_t'\partial_s\ln p(2s_t, \theta',z)p(2s_t,\theta',z)\rho^p(z,y)\,dz\nu(dy)\right)\mu_{\b_t,s_t}(\theta')\,d\theta'.
 \end{split}
 \end{equation}
 It is known that there exists $C_0>0$ such that $\forall s\in(0,1]$
 \begin{equation}
 \label{FEP9bis}
 |\partial_s\ln p(s,\theta, z)|\le \f{C_0}{s^{2}},
 \end{equation}
 see e.g. \cite{Hsu:99} and~\cite{Sheu:91} where  bounds of the type $\di |\n_\theta\ln p(s,\theta,z)|\le \f{C_1}{s}$ and  $\di |\n^2_\theta\ln p(s,\theta,z)|\le \f{C_2}{s^2}$ are given. Here we use  $$|\partial_s\ln p(s,\theta,z)|=\f12\left|\f{\Delta_\theta p(s,\theta,z)}{p(s,\theta,z)}\right|\le \f{{\rm dim}M}{2}\left(|\n_\theta^2\ln p(s,\theta,z)|+|\n_\theta \ln p(s,\theta,z)|^2\right). $$
 So~\eqref{FEP9} and~\eqref{FEP9bis} yield
 \begin{equation}
 \label{FEP9ter1}
 \begin{split}
 |\partial_t\ln(\mu_{\b_t,s_t}(\theta))|\le D^p\left(4\b_t'+\f{C_0\b_t|s_t'|}{s_t^{2}}\right).
 \end{split}
 \end{equation}
 which implies 
 \begin{equation}
 \label{FEP9ter}
 \begin{split}
 |\partial_t\ln(\mu_{\b_t,s_t}(\theta))|\le C\left(\b_t'+\f{\b_t|s_t'|}{s_t^{2}}\right).
 \end{split}
 \end{equation}
 with
 \begin{equation}
 \label{FEP9ter2}
 \begin{split}
 C=D^p(4+C_0).
 \end{split}
 \end{equation}
 Evaluating with~\eqref{FEP0} and integrating on $M$ we get 
 \begin{equation}
 \label{FEP9quad}
 \left|-\int_M\partial_t\ln(\mu_{\b_t,s_t}(\theta))m_t(\theta)\,d\theta\right|\le \f{C}{(1+t)k}\left(1+\ln(1+t)\right).
 \end{equation}
 Now we split the second term in the right of~\eqref{FEP8} into 
 \begin{equation}
 \label{FEP10}
 \begin{split} &\int_ML_t\left[\ln\left(\f{m_t(\theta)}{\mu_{\b_t,s_t}(\theta)}\right)\right]m_t(\theta,y)\,d\theta dy\\
 &=\int_ML_t^{0}\left[\ln\left(\f{m_t(\theta)}{\mu_{\b_t,s_t}(\theta)}\right)\right]m_t(\theta)\,d\theta
 +\int_MR_t(\theta,y)\left[\ln\left(\f{m_t(\theta)}{\mu_{\b_t,s_t}(\theta)}\right)\right]m_t(\theta,y)\,d\theta dy.
 \end{split}
 \end{equation}
 We have as for~\eqref{SA11}
 \begin{equation}
 \label{FEP11}
 \begin{split} \int_ML_t^{0}\left[\ln\left(\f{m_t(\theta)}{\mu_{\b_t,s_t}(\theta)}\right)\right]m_t(\theta)\,d\theta&=-2\int_M\left\|\n\sqrt{\f{m_t(\theta)}{\mu_{\b_t,s_t}(\theta)}}\right\|^2\mu_{\b_t,s_t}(\theta)\,d\theta\\
 &\le - 2c_2(\b_t\vee 1)^{-p}\exp\left(-c(U_{2s_t})\b_t\right)J_t
 \end{split}
 \end{equation}
 for some $c_2>0$ and integer $p>0$
 by logarithmic Sobolev inequality (\cite{Miclo:92}).
 
 The computation for the second term is similar to the one after~\eqref{SA11} and we get
 \begin{align*} &\int_MR_t(\theta,y)\left[\ln\left(\f{m_t(\theta)}{\mu_{\b_t,s_t}(\theta)}\right)\right]m_t(\theta,y)\,d\theta dy\\
 &=2\b_t\int_M\sqrt{\f{\mu_{\b_t,s_t}}{m_t}(\theta)}\left\langle d\sqrt{\f{m_t}{\mu_{\b_t,s_t}}(\theta)},R_t(\theta)\right\rangle m_t(\theta)\,d\theta
 \end{align*}
 with 
 $$
 R_t(\theta)=-\int_M\grad_{\theta}\kappa_{s_t}(\cdot, y)(m_t(y|\theta)-\nu_{s_t}(y))\,dy,
 $$
 and again
 \begin{align*} &\int_MR_t(\theta,y)\left[\ln\left(\f{m_t(\theta)}{\mu_{\b_t,s_t}(\theta)}\right)\right]m_t(\theta)\,d\theta\\
 &\le \b_t^2\int_M\|R_t(\theta)\|^2m_t(\theta)\,d\theta+\int_M\left\|\n\sqrt{\f{m_t}{\mu_{\b_t,s_t}}(\theta)}\right\|^2\mu_{\b_t,s_t}(\theta)\,d\theta.
 \end{align*}
 Replacing $2c_2$ by $c_2$ in~\eqref{FEP11} we can after summing get rid of the second term in the right. 
 
 Here again
 \begin{align*} \int_M\|R_t(\theta)\|^2m_t(\theta)\,d\theta&\le 32K^2I_t
 \end{align*}
 where we have defined 
 \begin{equation}
 \label{FEP12}
 I_t=\int_{M\times M}\ln\left(\f{m_t(y|\theta)}{\nu_{s_t}(y)}\right)m_t(y,\theta)\,dy.
 \end{equation}
 At this stage we proved that 
 \begin{equation}
 \label{FEP13}
 \begin{split}
 \f{dJ_t}{dt}&\le \f{C}{(1+t)k}\left(1+\ln(1+t)\right)\\&-c_2(\b_t\vee 1)^{-p}\exp\left(-c(U_{2s_t})\b_t\right)J_t+\b_t^2 32K^2I_t.
 \end{split}
 \end{equation}
 
 The next step is to find a suitable bound for $\f{dI_t}{dt}$. As before 
 \begin{equation}
 \label{FEP14}
 \begin{split}
 \f{dI_t}{dt}&=-\int_M\partial_t\ln(\nu_{s_t}(y))m_t(\theta,y)\,d\theta dy\\&+\int_{M\times M}L_t\left[\ln\left(\f{m_t(y|\theta)}{\nu_{s_t}(y)}\right)\right]m_t(y,\theta)\,d\theta dy
 \end{split}
 \end{equation}
 and 
 \begin{equation}
 \label{FEP14ter}
 \begin{split}
 &\left|\int_M\partial_t\ln(\nu_{s_t}(y))m_t(\theta,y)\,d\theta dy\right|\\
 &=\left|s_t'\int_M\partial_s\ln p(s_t,y,z)p(s_t,y,z)\,\nu(dz)\right|\\
 &\le \f{|s_t'|C_0}{s_t^{2}}=\f{C_0}{1+t}.
 \end{split}
 \end{equation}
  Now
 \begin{equation}
 \label{FEP14bis}
 \begin{split}
 &\int_{M\times M}L_t\left[\ln\left(\f{m_t(y|\theta)}{\nu_{s_t}(y)}\right)\right]m_t(y,\theta)\,d\theta dy\\&
 =\int_{M\times M}(L_{2,t}+L_{1,t})\left[\ln\left(\f{m_t(y|\theta)}{\nu_{s_t}(y)}\right)\right]m_t(y,\theta)\,d\theta dy.
 \end{split}
 \end{equation}
 We begin with the first term:
 \begin{align*}
 &\int_{M\times M}L_{2,t}\left[\ln\left(\f{m_t(y|\theta)}{\nu_{s_t}(y)}\right)\right]m_t(y,\theta)\,d\theta dy\\
 &=\g_t^{-1}\int_{M\times M}\ln\left(\f{m_t(y|\theta)}{\nu_{s_t}(y)}\right)\left(\nu_{s_t}(y)-m_t(y|\theta)\right)m_t(\theta)\,d\theta
 \end{align*}
 and estimate it as for~\eqref{SA15}:
\begin{equation}
\label{FEP15}
\int_{M\times M}L_{2,t}\left[\ln\left(\f{m_t(y|\theta)}{\nu_{s_t}(y)}\right)\right]m_t(y,\theta)\,d\theta dy\le -\g_t^{-1}I_t.
\end{equation}
For the second term in the right of~\eqref{FEP14} we need to introduce the density $f_t(y)$ of $Y_t^{s_t}$. Since 
$$
\ln\left(\f{m_t(y|\theta)}{\nu_{s_t}(y)}\right)=\ln\left(\f{m_t(\theta|y)}{m_t(\theta)}\right)+ \ln\left(\f{f_t(y)}{\nu_{s_t}(y)}\right)
$$
it rewrites as
\begin{equation}
\label{FEP16}
\begin{split}
&\int_{M\times M}L_{1,t}\left[\ln\left(\f{m_t(\theta|y)}{m_t(\theta)}\right)\right]
m_t(y,\theta)\,d\theta dy\\
&=\int_{M\times M}L_{1,t}\left[\ln (m_t(\theta|y))-\ln (m_t(\theta))\right]
m_t(y,\theta)\,d\theta dy.
\end{split}
\end{equation}
Similarly to~\eqref{SA17}
\begin{equation}
\label{FEP17}
\begin{split}
&\int_{M\times M}L_{1,t}\ln (m_t(\theta|y))
m_t(y,\theta)\,d\theta dy\\
&=\f12 \int_{M\times M}\D\ln (m_t(\theta|y))-\b_t\int_M\left\langle d_\theta \ln m_t(\cdot|y),\grad_{\theta}\kappa_{s_t}(\cdot, y)\right\rangle m_t(y,\theta)\,d\theta dy\\
&=-2\int_{M\times M}\left\|\n\sqrt{m_t(\theta|y)}\right\|^2\,d\theta f_t(y)dy\\&-\b_t\int_{M\times M}\left\langle d_\theta m_t(\cdot|y),\grad_{\theta}\kappa_{s_t}(\cdot, y)\right\rangle 
f_t(y)\,d\theta dy.
\end{split}
\end{equation}
For the absolute value of the last term:
\begin{equation}
\label{FEP18}
\begin{split}
&\left|-\b_t\int_{M\times M}\left\langle d_\theta m_t(\cdot|y),\grad_{\theta}\kappa(\cdot, y)\right\rangle \nu(y)\,
f_t(y)\,d\theta dy\right|\\&\le\f12\b_t^2K^2+2\int_{M\times M}\left\|\n \sqrt{m_t(\theta|y)}\right\|^2f_t(y) \,d\theta dy.
\end{split}
\end{equation}
We get as in~\eqref{SA19}
\begin{equation}
\label{FEP19}
\int_{M\times M}L_{1,t}\ln (m_t(\theta|y))
m_t(y,\theta)\,d\theta dy\le \f12\b_t^2K^2
\end{equation}
Then we bound the  last term in~\eqref{FEP16}: 
\begin{equation}
\label{FEP20}
\begin{split}
-L_{1,t} \ln(m_t(\theta))=-L_{1,t}\ln\left(\f{m_t(\theta)}{\mu_{\b_t,s_t}(\theta)}\right)-L_{1,t}\ln(\mu_{\b_t,s_t}(\theta)).
\end{split}
\end{equation}
We already know by~\eqref{FEP8} and~\eqref{FEP9quad} that 
\begin{equation}
\label{FEP21}
-\int_ML_{1,t}\ln\left(\f{m_t(\theta)}{\mu_{\b_t,s_t}(\theta)}\right)m_t(\theta)\,d\theta\le -\f{dJ_t}{dt}+\f{C}{(1+t)k}\left(1+\ln(1+t)\right).
\end{equation}

For the second term we have 
\begin{equation}
\label{FEP22}
\begin{split}
L_{1,t}\ln(\mu_{\b_t,s_t}(\theta))&=-2\b_tL_{1,t}U_{2s_t}(\theta)\\
&=-\b_t\D U_{2s_t}(\theta)+2\b_t^2\left\langle dU_{2s_t}, \grad_{\theta}\kappa_{s_t}(\cdot, y)\right\rangle\\
&\le K'(\b_t\vee 1)\b_ts_t^{-2}
\end{split}
\end{equation}
for some $K'>0$, where we used
$$
\D U_{2s}=\int_M\left(\D_\theta\ln p(2s,\theta,y)+\|\n_\theta\ln p(2s,\theta,y)\|^2 \right)p(2s,\theta,y)\rho^p(y,z)\,\nu(dz)
$$ and standard bound for the first and second derivatives of the heat kernel (\cite{Hsu:99} and \cite{Sheu:91}). 
%with 
%\begin{equation}
%\label{FEP23bis}
%K'=\sup_{\theta,y\in M}|\D_\theta\kappa(\cdot,y)|+K^2.
%\end{equation}

Finally we obtain 
\begin{equation}
\label{FEP23}
\begin{split}
\f{dI_t}{dt}\le \f{C_0}{(1+t)}+\f{K'}{k^2}(\ln(1+t)\vee k)(\ln(1+t))^3-\f{dJ_t}{dt}-(1+t)I_t
\end{split}
\end{equation}
together with~\eqref{FEP13}:
\begin{equation}
\label{FEP13bis}
 \f{dJ_t}{dt}\le \f{C}{(1+t)k}\left(1+\ln(1+t)\right)-c_2(\b_t\vee 1)^{-p}\exp\left(-c(U_{2s_t})\b_t\right)J_t+2\b_t^232K^2I_t
 \end{equation}
 which rewrites as 
 \begin{equation}
\label{FEP13ter}
\begin{split}
\f{dI_t}{dt}&\le k_1(\ln(1+t))^{4}-\f{dJ_t}{dt}-(1+t)I_t
\end{split}
\end{equation}
and
\begin{equation}
\label{FEP13quad}
 \f{dJ_t}{dt}\le c_1\left(\f{\ln(1+t)}{1+t}+(\ln(1+t))^{2}I_t\right)-c_2(\ln(1+t))^{-p}(1+t)^{-\f{c(U_{2s_t})}{k}}J_t
 \end{equation}
 for some constants $c_1, k_1>0$, as soon as $t\ge 2$.
 At this stage we can use a similar computation to the end of the proof of theorem 1 in \cite{Miclo:96} to obtain that under assumptions~\eqref{FEP0} and~\eqref{SA0ter} then 
 \begin{equation}
 \label{FEP24}
 \lim_{t\to\infty} J_t=0.
 \end{equation}
 However we will do the calculation for completeness, and because there are some small differences.
 Recall $U_s\to U$ uniformly as $s\to 0$. Moreover $2s_t\to 0$ as $t\to\infty$, so we get
 $$
\limsup_{t\to\infty}c(U_{2s_t})\le c(U).
 $$
 As a consequence, for $t$ sufficiently large we have 
 \begin{equation}
 \label{FEP33}
 \f{c(U_{2s_t})}{k}\le 1-\e
 \end{equation}
 for some $\e>0$.
 Let 
 \begin{equation}
 \label{FEP25}
 \ell_t=\f{c_1(\ln(1+t))^2}{1+t+c_1(\ln(1+t))^2-c_2(\ln(1+t))^{-p}(1+t)^{-(1-\e)}}
 \end{equation}
 where $\e>0$ is defined in~\eqref{FEP33}.
 It is easily checked that for $t$ sufficiently large  $\ell_t$ is positive and decreasing, and that it converges to $0$ as $t\to\infty$.
 Define
 \begin{equation}
 \label{FEP26}
 K_t=J_t+\ell_t I_t.
 \end{equation}
 We will prove that $K_t\to 0$ as $t\to \infty$ and from this we will get~\eqref{FEP24}. 
 
 for $t$ sufficiently large, 
 \begin{equation}
 \label{FEP27}
 \f{dK_t}{dt}\le \f{dJ_t}{dt}+\ell_t\f{dI_t}{dt}
 \end{equation}
 and this yields with~\eqref{FEP13ter} and~\eqref{FEP13quad} 
 \begin{align*}
 \f{dK_t}{dt}&\le (1-\ell_t)c_1\f{\ln(1+t)}{1+t}+c_1(\ln(1+t))^{2}I_t\\
 &-\ell_tc_1(\ln(1+t))^{2}I_t-(1-\ell_t)c_2(\ln(1+t))^{-p}(1+t)^{-\f{c(U_{2s_t})}{k}}J_t\\
 &+\ell_tk_1(\ln(1+t))^{4}-(1+t)\ell_tI_t.
 \end{align*}
 
 Replacing $c_1(\ln(1+t))^{2}$ at the end of the first line by
 $$
 \ell_t\left(1+t+c_1(\ln(1+t))^2-c_2(\ln(1+t))^{-p}(1+t)^{-(1-\e)}\right)
 $$
 by the help of~\eqref{FEP25} we obtain
 \begin{align*}
 \f{dK_t}{dt}&\le c_1\f{\ln(1+t)}{1+t}-c_2\ell_t(\ln(1+t))^{-p}(1+t)^{-\f{c(U_{2s_t})}{k}}I_t\\
 & -(1-\ell_t)c_2(\ln(1+t))^{-p}(1+t)^{-\f{c(U_{2s_t})}{k}}J_t+\ell_tk_1(\ln(1+t))^{4}
 \end{align*}
 and this yields
 using 
$\di
-(1+t)^{-\f{c(U_{2s_t})}{k}}\le -(1+t)^{-(1-\e)}
$:
 \begin{equation}
 \label{FEP28}
 \f{dK_t}{dt}\le A_t-B_tK_t
 \end{equation}
 with 
 \begin{equation}
 \label{FEP29}
 A_t=c_1\f{\ln(1+t)}{1+t}+\ell_tk_1\ln(1+t))^{4}
 \end{equation}
 and 
 \begin{equation}
 \label{FEP30}
 B_t=(1-\ell_t)c_2(\ln(1+t))^{-p}(1+t)^{-(1-\e)}.
 \end{equation}
 A sufficient condition for $K_t$ to converge to $0$ as $t\to\infty$ is 
 \begin{equation}
 \label{FEP31}
 \int_\cdot^\infty B_t\,dt=+\infty
 \end{equation}
 and
 \begin{equation}
 \label{FEP32}
 \lim_{t\to\infty}\f{A_t}{B_t}=0.
 \end{equation}
  Condition~\eqref{FEP31} clearly is realized. As for condition~\eqref{FEP32} we easily see that
 $$
 \f{c_1\f{\ln(1+t)}{1+t}}{(1-\ell_t)c_2(\ln(1+t))^{-p}(1+t)^{-(1-\e)}}\to 0
 $$
 and also
 $$ \f{\ell_tk_1(\ln(1+t))^{4}}{(1-\ell_t)c_2(\ln(1+t))^{-p}(1+t)^{-(1-\e)}}\to 0
 $$
 from the fact that 
 $$
 \ell_t\le \f{c(\ln(1+t))^{2}}{1+t}
 $$
 for some $c>0$.
\end{proof}
%%%%%%%%%%%%%%%%%%%%%%%%%%%%%%%%%%%%%%%%%%%%%%%%%%%%%%%%%%%%%%%%%%%%%%%%%
%
%   R E F E R E N C E S
%
%%%%%%%%%%%%%%%%%%%%%%%%%%%%%%%%%%%%%%%%%%%%%%%%%%%%%%%%%%%%%%%%%%%%%%%%%

\providecommand{\bysame}{\leavevmode\hbox to3em{\hrulefill}\thinspace}

\end{document}